\documentclass[a4paper,twoside,fleqn,notitlepage]{article}%
\usepackage{amsmath}
\usepackage{amsfonts}
\usepackage{amssymb}
\usepackage{graphicx}%
\providecommand{\U}[1]{\protect\rule{.1in}{.1in}}
\newtheorem{theorem}{Theorem}

\newtheorem{corollary}[theorem]{Corollary}

\newtheorem{definition}[theorem]{Definition}
\newtheorem{example}[theorem]{Example}

\newtheorem{lemma}[theorem]{Lemma}

\newtheorem{proposition}[theorem]{Proposition}
\newtheorem{remark}[theorem]{Remark}

\newenvironment{proof}[1][Proof]{\noindent\textbf{#1.} }{\ \rule{0.5em}{0.5em}}
\begin{document}

\title{Unification of $q-$exponential function and related $q-$numbers and polynomials}
\author{N. I. Mahmudov and M. Momenzadeh\\Eastern Mediterranean University\\Gazimagusa, TRNC, Mersin 10, Turkey\\Near East University\\Nicosia, TRNC, Mersin 10, Turkey \\Email: nazim.mahmudov@emu.edu.tr\\\ \ \ \ \ \ \ \ \ mohammad.momenzadeh@neu.edu.tr}
\maketitle

\begin{abstract}
The main purpose of this paper is to introduce and investigate a class of
generalized Bernoulli polynomials and Euler polynomials based on the
generating function. we unify all forms of $q-$exponential functions by one
more parameter. we study some conditions on this parameter to related this to
some classical results for $q$-Bernoulli numbers and polynomials.

\end{abstract}

\section{Introduction}

\bigskip In combinatorial mathematics, a $q$-exponential is a $q$-analog of
the exponential function, namely the eigenfunction of a $q$-derivative. There
are many $q$-derivatives, for example, the classical $q$-derivative, the
Askey-Wilson operator, etc. \cite{eXton}.Therefore, unlike the classical
exponentials, q-exponentials are not unique. In the standard approach to the
$q-$calculus, two exponential function are used.These $q-$exponentials are
defined by%

\begin{align*}
e_{q}\left(  z\right)   &  =\sum_{n=0}^{\infty}\frac{z^{n}}{\left[  n\right]
_{q}!}=\prod_{k=0}^{\infty}\frac{1}{\left(  1-\left(  1-q\right)
q^{k}z\right)  },\ \ \ 0<\left\vert q\right\vert <1,\ \left\vert z\right\vert
<\frac{1}{\left\vert 1-q\right\vert },\ \ \ \ \ \ \ \\
E_{q}(z)  &  =e_{1/q}\left(  z\right)  =\sum_{n=0}^{\infty}\frac{q^{\frac
{1}{2}n\left(  n-1\right)  }z^{n}}{\left[  n\right]  _{q}!}=\prod
_{k=0}^{\infty}\left(  1+\left(  1-q\right)  q^{k}z\right)
,\ \ \ \ \ \ \ 0<\left\vert q\right\vert <1,\ z\in\mathbb{C},
\end{align*}
In addition, The improved $q$-exponential function is defined by
\cite{improved}%

\[
\mathcal{E}_{q}\left(  z\right)  =e_{q}\left(  \frac{z}{2}\right)  E_{q}%
(\frac{z}{2})=\sum_{n=0}^{\infty}\frac{z^{n}}{\left[  n\right]  _{q}!}%
\frac{\left(  -1,q\right)  _{n}}{2^{n}}=\prod_{k=0}^{\infty}\frac{\left(
1+\left(  1-q\right)  q^{k}\frac{z}{2}\right)  }{\left(  1-\left(  1-q\right)
q^{k}\frac{z}{2}\right)  },\ \ \ 0<\left\vert q\right\vert <1,\ \left\vert
z\right\vert <\frac{2}{\left\vert 1-q\right\vert },\ \ \ \ \ \ \
\]
The Bernoulli numbers $\left\{  B_{m}\right\}  _{m\geq0}$ are rational numbers
in a sequence defined by the binomial recurrence formula
\begin{equation}
\sum_{k=0}^{m}\left(
\begin{array}
[c]{c}%
m\\
k
\end{array}
\right)  B_{k}-B_{m}=\left\{
\begin{tabular}
[c]{ll}%
$1,$ & $m=1,$\\
$0,$ & $m>1,$%
\end{tabular}
\ \ \ \ \ \ \ \ \ \right.  \label{b1}%
\end{equation}
or equivalently, the generating function%
\[
\sum_{k=0}^{\infty}B_{k}\frac{t^{k}}{k!}=\frac{t}{e^{t}-1}.
\]
The $q$-binomial formula is known as%
\[
\left(  1-a\right)  _{q}^{n}=\left(  a;q\right)  _{n}=%
{\displaystyle\prod\limits_{j=0}^{n-1}}
\left(  1-q^{j}a\right)  =\sum_{k=0}^{n}\left(
\begin{array}
[c]{c}%
n\\
k
\end{array}
\right)  _{q}q^{\frac{1}{2}k\left(  k-1\right)  }\left(  -1\right)  ^{k}a^{k}.
\]
The above $q$-standard notation can be found in \cite{andrew}.

Over 70 years ago, Carlitz extended the classical Bernoulli and Euler numbes
and polynomials, and introduced the $q$-Bernoulli and the $q$-Euler numbers
and polynomials (see \cite{calitz1}, \cite{calitz2} and \cite{calitz3} ).
There are numerous recent investigations on this subject. (\cite{cenkci1},
\cite{cenkci2}, \cite{cenkci3}, \cite{choi2} and \cite{choi3}), Srivastava
\cite{sri1}, Srivastava et al. \cite{sri}. The main part of these
generalizations is the definition of q-analogue of exponential function. By
defining the suitable q-analogue of exponential function, they derive to the
different definitions for q-Bernoulli numbers. In this case some interesting
properties are discovered.\cite{mohammad}. The unification of q-exponential is
introduced in the next definition. This function depends on the parameter and
by changing this parameter we can reach to the different versions of
q-exponential function. 

\begin{definition}
we define unification of q-exponential function as follow%

\[
\mathcal{E}_{q,\alpha}\left(  z\right)  =\sum_{n=0}^{\infty}\frac{z^{n}%
}{\left[  n\right]  _{q}!}\alpha(q,n)
\]

where $z$ is any complex number and $\alpha(q,n)$ is a function of $q$ and
$n$. In addition, $\alpha(q,n)$ approches to 1, where $q$ tends one from the
left side.In the special case where $\alpha(q,n)=1$, and $\alpha
(q,n)=q^{\binom{n}{2}}$ we reach to $e_{q}\left(  z\right)  $ and $E_{q}(z)$ respectively.
\end{definition}

At the next lemma, we will discuss about the conditions that make
$\mathcal{E}_{q,\alpha}\left(  z\right)  $ convergent.There are some
restrictions, that has to be considered. since $\mathcal{E}_{q,\alpha}\left(
z\right)  $ is the $q$-analogue of exponential function, $\alpha(q,n)$
approches to 1, where $q$ tends one from the left side. For the rest of the
paper we will denote $\alpha(q,n)$ by $\alpha_{n}$, however we keep this in
our mind that $\alpha(q,n)$ is depend on $q$ and $n$.

\begin{lemma}
If $\lim\left\vert \frac{\alpha_{n+1}}{\left[  n+1\right]  _{q}\alpha_{n}%
}\right\vert $ does exist as n tends infinity and is equal to $l$, Then
$q-$exponential function $\mathcal{E}_{q,\alpha}\left(  z\right)  $ is
analytic in the disc $|z|<\left(  l\right)  ^{-1}.$

\begin{proof}
In order to obtain the radius of convergence, we compute%

\[
\lim_{n\rightarrow\infty}\left\vert \frac{z^{n+1}\alpha_{n+1}}{\left[
n+1\right]  _{q}!}\right\vert \left\vert \frac{\left[  n\right]  _{q}!}%
{z^{n}\alpha_{n}}\right\vert =\lim_{n\rightarrow\infty}\left\vert \frac
{\alpha_{n+1}}{\left[  n+1\right]  _{q}\alpha_{n}}\right\vert |z|
\]

Then, using d'Alembert's test, we get (for $q\neq1$) the radius of convergence.

\begin{example}
Let $\alpha_{n}$ is equal to $1,$ $q^{\binom{n}{2}},$ $\frac{(-1,q)_{n}}%
{2^{n}},$ then we reach to $e_{q}(z),$ $E_{q}(z)$ and improved $q-$exponential
function $\mathcal{E}_{q}\left(  z\right)  $ \cite{improved} repectively. Then
the radius of convergence becomes $\frac{1}{|1-q|},$ infinity and $\frac
{2}{|1-q|}$ repectively where $0<|q|<1.$

With this $q-$exponential function, we define the new class of $q-$Bernoulli
numbers and polynomials. Next definition denotes a general class of these new
$q-$numbers and polynomials.
\end{example}
\end{proof}
\end{lemma}

\begin{definition}
\label{D:1}Let $q\in\mathbb{C},\ 0<\left\vert q\right\vert <1.$ The
$q$-Bernoulli numbers $\mathfrak{B}_{n,q,\alpha}$ and polynomials
$\mathfrak{B}_{n,q,\alpha}\left(  x,y\right)  $ and $q$-Euler numbers
$\mathfrak{E}_{n,q,\alpha}$ and polynomials $\mathfrak{E}_{n,q,\alpha}\left(
x,y\right)  $ and The $q$-Genocchi numbers $\mathfrak{G}_{n,q,\alpha}$ and
polynomials $\mathfrak{G}_{n,q,\alpha}\left(  x,y\right)  $ in two variables
$x,y$ respectively are defined by the means of the generating functions:%
\begin{align}
\widehat{\mathfrak{B}}\left(  t\right)   &  =\frac{t}{\mathcal{E}_{q,\alpha
}\left(  t\right)  -1}=\sum_{n=0}^{\infty}\mathfrak{B}_{n,q,\alpha}\frac
{t^{n}}{\left[  n\right]  _{q}!},\ \ \ \left\vert t\right\vert <2\pi
,\label{a1}\\
\frac{t}{\mathcal{E}_{q,\alpha}\left(  t\right)  -1}\mathcal{E}_{q,\alpha
}\left(  tx\right)  \mathcal{E}_{q,\alpha}\left(  ty\right)   &  =\sum
_{n=0}^{\infty}\mathfrak{B}_{n,q,\alpha}\left(  x,y\right)  \frac{t^{n}%
}{\left[  n\right]  _{q}!},\ \ \ \left\vert t\right\vert <2\pi,\nonumber\\
\frac{2}{\mathcal{E}_{q,\alpha}\left(  t\right)  +1}  &  =\sum_{n=0}^{\infty
}\mathfrak{E}_{n,q,\alpha}\frac{t^{n}}{\left[  n\right]  _{q}!}%
,\ \ \ \left\vert t\right\vert <\pi,\nonumber\\
\frac{2}{\mathcal{E}_{q,\alpha}\left(  t\right)  +1}\mathcal{E}_{q,\alpha
}\left(  tx\right)  \mathcal{E}_{q,\alpha}\left(  ty\right)   &  =\sum
_{n=0}^{\infty}\mathfrak{E}_{n,q,\alpha}\left(  x,y\right)  \frac{t^{n}%
}{\left[  n\right]  _{q}!},\ \ \ \left\vert t\right\vert <\pi\nonumber\\
\frac{2t}{\mathcal{E}_{q,\alpha}\left(  t\right)  +1}  &  =\sum_{n=0}^{\infty
}\mathfrak{G}_{n,q,\alpha}\frac{t^{n}}{\left[  n\right]  _{q}!}%
,\ \ \ \left\vert t\right\vert <\pi,\nonumber\\
\frac{2t}{\mathcal{E}_{q,\alpha}\left(  t\right)  +1}\mathcal{E}_{q,\alpha
}\left(  tx\right)  \mathcal{E}_{q,\alpha}\left(  ty\right)   &  =\sum
_{n=0}^{\infty}\mathfrak{G}_{n,q,\alpha}\left(  x,y\right)  \frac{t^{n}%
}{\left[  n\right]  _{q}!},\ \ \ \left\vert t\right\vert <\pi.\nonumber
\end{align}

\end{definition}

If the convergence conditions are hold for q-exponential function, It is
obvious that by tending $q$ to 1 from the left side, we lead to the classic
definition of these polynomials.we mention that $\alpha(q,n)$ is respect to
$q$ and $n.$ In addition by tending $q$ to 1$^{-},\mathcal{E}_{q,\alpha
}\left(  z\right)  $ approach to the ordinary exponential function. that
means:%
\begin{align*}
\mathfrak{B}_{n,q,\alpha}  &  =\mathfrak{B}_{n,q,\alpha}\left(  0\right)
,\ \ \ \lim_{q\rightarrow1^{-}}\mathfrak{B}_{n,q}\left(  x,y\right)
=B_{n}\left(  x+y\right)  ,\ \ \ \lim_{q\rightarrow1^{-}}\mathfrak{B}%
_{n,q}=B_{n},\\
\mathfrak{E}_{n,q,\alpha}  &  =\mathfrak{E}_{n,q,\alpha}\left(  0\right)
,\ \ \ \lim_{q\rightarrow1^{-}}\mathfrak{E}_{n,q}\left(  x,y\right)
=E_{n}\left(  x+y\right)  ,\ \ \ \lim_{q\rightarrow1^{-}}\mathfrak{E}%
_{n,q}=E_{n},\\
\mathfrak{G}_{n,q,\alpha}  &  =\mathfrak{G}_{n,q,\alpha}\left(  0\right)
,\ \ \ \lim_{q\rightarrow1^{-}}\mathfrak{G}_{n,q}\left(  x,y\right)
=G_{n}\left(  x+y\right)  ,\ \ \ \lim_{q\rightarrow1^{-}}\mathfrak{G}%
_{n,q}=G_{n}.
\end{align*}
Here $B_{n}\left(  x\right)  ,$ $E_{n}\left(  x\right)  $ and $G_{n}\left(
x\right)  $ denote the classical Bernoulli, Euler and Genocchi polynomials
which are defined by%
\[
\frac{t}{e^{t}-1}e^{tx}=\sum_{n=0}^{\infty}B_{n}\left(  x\right)  \frac{t^{n}%
}{n!},\ \ \ \text{\ }\frac{2}{e^{t}+1}e^{tx}=\sum_{n=0}^{\infty}E_{n}\left(
x\right)  \frac{t^{n}}{n!}\ \ \text{and\ \ \ \ }\frac{2t}{e^{t}+1}e^{tx}%
=\sum_{n=0}^{\infty}G_{n}\left(  x\right)  \frac{t^{n}}{n!}.
\]

The aim of the present paper is to obtain some results for the above newly
defined $q$-Bernoulli and $q$-Euler polynomials. In the next section we will
discuss about some restriction for $\alpha(q,n)$, such that the familar
results discovered. we will focus on two main properties of $q-$exponential
function, first in which situation $\mathcal{E}_{q,\alpha}\left(  z\right)
=\mathcal{E}_{q^{-1},\alpha}\left(  z\right)  ,$ second we investigate the
conditions for $\alpha(q,n)$ such that $\mathcal{E}_{q,\alpha}\left(
-z\right)  =\left(  \mathcal{E}_{q,\alpha}\left(  z\right)  \right)  ^{-1}.$A
lot of classical results are found by these two properties.The form of new
type of $q$-exponential function, motivate us to define a new $q-$addition and
$q-$substraction like a Daehee formula as follow%
\begin{align*}
\left(  x\oplus_{q}y\right)  ^{n} &  :=\sum_{k=0}^{n}\left(
\begin{array}
[c]{c}%
n\\
k
\end{array}
\right)  _{q}\alpha(q,k)x^{k}y^{n-k},\ \ \ n=0,1,2,...,\\
\left(  x\ominus_{q}y\right)  ^{n} &  :=\sum_{k=0}^{n}\left(
\begin{array}
[c]{c}%
n\\
k
\end{array}
\right)  _{q}\alpha(q,k)x^{k}\left(  -y\right)  ^{n-k},\ \ \ n=0,1,2,...
\end{align*}

\section{New exponential function and its properties}

In this section we shall provide some conditions on $\alpha(q,n)$ to reach two
main properties. first we try to find out, in which situation $\mathcal{E}%
_{q,\alpha}\left(  z\right)  =\mathcal{E}_{q^{-1},\alpha}\left(  z\right)  .$
This condition make $q-$exponential symmetry to $q$ factor and the properties
of related $q-$numbers will be preserved even if we change $q$ to $q^{-1}$.
Second property is the condition on $q$-exponential to reach multiplicative
inverse i.e. $\mathcal{E}_{q,\alpha}\left(  -z\right)  =\left(  \mathcal{E}%
_{q,\alpha}\left(  z\right)  \right)  ^{-1}.$ this property make the odd
coefficient of $q$-Bernoulli numbers exactly zero and related them to
$q-$trigonometric functions.

\begin{lemma}
\bigskip The new $q$-exponential function $\mathcal{E}_{q,\alpha}\left(
z\right)  $ satisfy $\mathcal{E}_{q,\alpha(q)}\left(  z\right)  =\mathcal{E}%
_{q^{-1},\alpha(q^{-1})}\left(  z\right)  ,$ if and only if $q^{\binom{n}{2}%
}\alpha(q^{-1},n)=\alpha(q,n)$.

\begin{proof}
The proof is based on the fact that $\left[  n\right]  _{q^{-1}}%
!=q^{-\binom{n}{2}}\left[  n\right]  _{q}!,$ therefore%
\[
\mathcal{E}_{q^{-1},\alpha(q^{-1})}\left(  z\right)  =\sum_{n=0}^{\infty}%
\frac{z^{n}}{\left[  n\right]  _{q^{-1}}!}\alpha(q^{-1},n)=\sum_{n=0}^{\infty
}\frac{z^{n}}{\left[  n\right]  _{q}!}\alpha(q,n)=\mathcal{E}_{q,\alpha
}\left(  z\right)
\]

\end{proof}
\end{lemma}

\begin{proof}
On the another hand, the another side of statement can be found by equating
the coefficient of above summation.
\end{proof}

\begin{corollary}
If $\alpha(q,n)$ is in a form of polynomial that means $\alpha(q,n)=\sum
_{i=0}^{m}a_{i}q^{i}$ , to satisfy $\mathcal{E}_{q,\alpha(q)}\left(  z\right)
=\mathcal{E}_{q^{-1},\alpha(q^{-1})}\left(  z\right)  $, we have%

\[
\deg(\alpha(q,n))=m=\binom{n}{2}-j\leqslant\binom{n}{2}\text{, and }%
a_{j+k}=a_{m+k}\text{ where }k=0,1,...,m-j
\]

where j is the leading index, such that $a_{j}\neq0$ and for $0\leqslant k<j,$
$a_{k}=0.$

\begin{proof}
First, we want to mention that $\sum_{i=0}^{m}a_{i}=1$, becuase $\alpha(q,n)$
approches to 1, where $q$ tends one from the left side. In addition as we
assumed $\alpha(q,n)=\sum_{i=0}^{m}a_{i}q^{i}$, by simple substitution
$q^{-1}$ instead of $q$, and $\sum_{i=0}^{m}a_{i}=1$ lead us to%

\[
q^{\binom{n}{2}-m}\sum_{i=0}^{m}a_{m-i}q^{i}=\sum_{i=0}^{m}a_{i}q^{i}%
\]

equating the coefficient of $q^{k}$ to reach the statement.

\begin{example}
simplest example of the previous corollary will be happened when
$\alpha(q,n)=q^{\frac{\binom{n}{2}}{2}}.$This case leads us to the following
exponential function%

\[
\mathcal{E}_{q,\alpha}\left(  z\right)  =\sum_{n=0}^{\infty}\frac{z^{n}%
}{\left[  n\right]  _{q}!}q^{\frac{\binom{n}{2}}{2}}\text{ }\&\text{
}\mathcal{E}_{q^{-1},\alpha(q^{-1})}\left(  z\right)  =\mathcal{E}_{q,\alpha
}\left(  z\right)
\]

The another example will be occured if \ $\alpha(q,n)=\frac{(-1,q)_{n}}{2^{n}%
}=\frac{(1+q)(1+q^{2})...(1+q^{n})}{2^{n-1}}.$By using $q$-binomial formula
$\alpha(q,n)=\frac{1}{2^{n}}\sum_{i=0}^{n}\left(
\begin{array}
[c]{c}%
n\\
i
\end{array}
\right)  _{q}q^{\frac{i(i-1)}{2}}.$As we expect, where $q$ tends $1$ from the
left side,$\alpha(q,n)$ approach to $1$. This presentation is not in a form of
previous corollary, However $q^{\binom{n}{2}}\frac{(1+q^{-1})(1+q^{-2}%
)...(1+q^{-n})}{2^{n-1}}=\alpha(q,n).$This parameter leads us to the improved
$q$-exponential function%

\[
\mathcal{E}_{q}\left(  z\right)  =\mathcal{E}_{q,\alpha}\left(  z\right)
=\sum_{n=0}^{\infty}\frac{z^{n}}{\left[  n\right]  _{q}!}\frac{(-1,q)_{n}%
}{2^{n}}\text{ }\&\text{ }\mathcal{E}_{q^{-1}}\left(  z\right)  =\mathcal{E}%
_{q}\left(  z\right)
\]

The properties of $q$-Bernoulli polynomials related to this improved
$q-$exponential function was studied at \cite{mohammad}.

\begin{remark}
It's obvious that if we substitute $q$ to $q^{-1}$ in any kind of
$q$-exponential function and achieve another $q-$analogue of exponential
function, the parameter $\alpha(q,n)$ will change to $\beta(q,n)$, and
$q^{\binom{n}{2}}\alpha(q^{-1},n)=\beta(q,n).$ The famous case is standard
$q-$exponential function:%

\begin{align*}
e_{q^{-1}}\left(  z\right)   &  =\mathcal{E}_{q^{-1},\alpha(q^{-1})}\left(
z\right)  =\sum_{n=0}^{\infty}\frac{z^{n}}{\left[  n\right]  _{q^{-1}}!}\text{
}\\
&  =\sum_{n=0}^{\infty}\frac{q^{\binom{n}{2}}z^{n}}{\left[  n\right]  _{q}%
!}=E_{q}\left(  z\right)  \text{ }\&\text{ }q^{\binom{n}{2}}\alpha
(q^{-1},n)=q^{\binom{n}{2}}=\beta(q,n)
\end{align*}

\end{remark}
\end{example}
\end{proof}
\end{corollary}

\begin{proposition}
\bigskip The new $q$-exponential function $\mathcal{E}_{q,\alpha}\left(
z\right)  $ satisfy $\mathcal{E}_{q,\alpha}\left(  -z\right)  =\left(
\mathcal{E}_{q,\alpha}\left(  z\right)  \right)  ^{-1},$ if and only if%

\[
\alpha(q,0)=1\text{ }\&\text{ }2\sum_{k=0}^{p-1}\left(
\begin{array}
[c]{c}%
n\\
k
\end{array}
\right)  _{q}(-1)^{k}\alpha_{k}\alpha_{n-k}=\left(
\begin{array}
[c]{c}%
n\\
p
\end{array}
\right)  _{q}(-1)^{p+1}\alpha_{p}^{2}\text{ where }n=2p\text{ and }p=1,2,...
\]

\begin{proof}
Since $\mathcal{E}_{q,\alpha}\left(  -z\right)  .\mathcal{E}_{q,\alpha}\left(
z\right)  =1$ has to be hold, we write the expansion for this equation.%

\[
\mathcal{E}_{q,\alpha}\left(  -z\right)  .\mathcal{E}_{q,\alpha}\left(
z\right)  =\sum_{n=0}^{\infty}\left(  \sum_{k=0}^{n}\left(
\begin{array}
[c]{c}%
n\\
k
\end{array}
\right)  _{q}(-1)^{k}\alpha_{k}\alpha_{n-k}\right)  \frac{z^{n}}{\left[
n\right]  _{q}!}=1
\]

Let call the expression on a bracket as $\beta_{k,q}.$ If $n$ is an odd
number, then%

\[
\beta_{n-k,q}=\left(
\begin{array}
[c]{c}%
n-k\\
k
\end{array}
\right)  _{q}(-1)^{n-k}\alpha(q,k)\alpha(q,n-k)=-\left(
\begin{array}
[c]{c}%
n\\
k
\end{array}
\right)  _{q}(-1)^{k}\alpha(q,k)\alpha(q,n-k)=-\beta_{k,q}\text{ where
}k=0,1,...,n
\]

Therefore for n as an odd number, we have the trivial equation. since $\left(
\begin{array}
[c]{c}%
n-k\\
k
\end{array}
\right)  _{q}=\left(
\begin{array}
[c]{c}%
n\\
k
\end{array}
\right)  _{q},$ The same discussion for even $n$ and equating $z^{n}%
-$coefficient togheter lead us to the proof.

\begin{remark}
The previous proposition can be rewritten as a system of nonlinear equations.
The following system shows a condition for $\alpha_{k}.$ we mention that
$\alpha_{k}\rightarrow1$ where $q\rightarrow1^{-}$ and $\alpha_{0}=1$.%
\begin{equation}
\left\{
\begin{array}
[c]{c}%
2\alpha_{2}\alpha_{1}-\binom{2}{1}_{q}\alpha_{0}\alpha_{0}=0\\
2\alpha_{4}\alpha_{1}-2\binom{4}{1}_{q}\alpha_{3}\alpha_{2}+\binom{4}{2}%
_{q}\alpha_{2}\alpha_{2}=0\\
2\alpha_{6}\alpha_{1}-2\binom{6}{1}_{q}\alpha_{5}\alpha_{2}+2\binom{6}{2}%
_{q}\alpha_{4}\alpha_{3}-2\binom{6}{3}_{q}\alpha_{3}\alpha_{3}=0\\
.\\
.\\
.\\
2\alpha_{n}\alpha_{1}-2\binom{n}{1}_{q}\alpha_{n-1}\alpha_{2}+2\binom{n}%
{2}_{q}\alpha_{n-2}\alpha_{3}-...+(-1)^{\frac{n}{2}}\binom{n}{\frac{n}{2}}%
_{q}\alpha_{\frac{n}{2}}\alpha_{\frac{n}{2}}=0
\end{array}
\right.  \label{a2}%
\end{equation}
For even $n$, we have $\frac{n}{2}$ equations and $n$ unknown variables. In
this case we can find $\alpha_{k}$ respect to $\frac{n}{2}$ parameters by the
recurence formula. For example, some few terms can be found as follow
\end{remark}

\begin{corollary}
\begin{remark}%
\[%
\begin{array}
[c]{c}%
\alpha_{0}=1\\
\alpha_{2}=\frac{1+q}{2}\frac{1}{\alpha_{1}}\\
\alpha_{4}=\frac{[4]_{q}}{2\alpha_{1}^{2}}\left(  [2]_{q}\alpha_{3}%
-\frac{[3]_{q}!}{4\alpha_{1}}\right) \\
\alpha_{6}=\binom{6}{1}_{q}+\binom{6}{3}_{q}-\frac{1}{2}\left(  \binom{6}%
{2}_{q}\left(  \frac{1+q}{2}\frac{1}{\alpha_{1}}\right)  \left(  \frac
{[4]_{q}}{2\alpha_{1}^{2}}\left(  [2]_{q}\alpha_{3}-\frac{[3]_{q}!}%
{4\alpha_{1}}\right)  \right)  \right)
\end{array}
\]

The familiar solution of this system is $\alpha(q,k)=\frac{\left(
-1,q\right)  _{k}}{2^{k}}.$ This $\alpha(q,k)$ leads us to the improved
exponential function. On the other hand, we can assume that all $\alpha_{k}$
for odd $k$ are $1$. Then by solving the system for these parameters, we reach
the another exponential function that satisfies $\mathcal{E}_{q,\alpha}\left(
-z\right)  =\left(  \mathcal{E}_{q,\alpha}\left(  z\right)  \right)  ^{-1}.$
\end{remark}
\end{corollary}
\end{proof}

\begin{lemma}
If $\frac{\alpha(q,n+1)}{\alpha(q,n)}$ can be demonstrated as a polynomial of
$q$, that means $\frac{\alpha(q,n+1)}{\alpha(q,n)}=$ $%
{\displaystyle\sum\limits_{k=0}^{m}}
a_{k}q^{k}$, then $D_{q}(\mathcal{E}_{q,\alpha}\left(  z\right)  )=%
{\displaystyle\sum\limits_{k=0}^{m}}
a_{k}\mathcal{E}_{q,\alpha}\left(  zq^{\frac{k}{n}}\right)  .$

\begin{proof}
The proof is based on the following identity%

\[
D_{q}(\mathcal{E}_{q,\alpha}\left(  z\right)  )=\sum_{n=1}^{\infty}%
\frac{z^{n-1}}{\left[  n-1\right]  _{q}!}\alpha(q,n)=\sum_{n=0}^{\infty}%
\frac{z^{n}}{\left[  n\right]  _{q}!}\left(  \alpha(q,n)%
{\displaystyle\sum\limits_{k=0}^{m}}
a_{k}q^{k}\right)  =%
{\displaystyle\sum\limits_{k=0}^{m}}
a_{k}\frac{\left(  zq^{\frac{k}{n}}\right)  ^{n}}{\left[  n\right]  _{q}%
!}\alpha(q,n)==%
{\displaystyle\sum\limits_{k=0}^{m}}
a_{k}\mathcal{E}_{q,\alpha}\left(  zq^{\frac{k}{n}}\right)  .
\]

\end{proof}

\begin{example}
For $\alpha(q,n)=1,$ $q^{\binom{n}{2}}$ and $\frac{\left(  -1,q\right)  _{n}%
}{2^{n}}$, the ratio of $\frac{\alpha(q,n+1)}{\alpha(q,n)}$ becomes $1,$
$q^{n}$ and $\frac{1+q^{n}}{2}$ respectively. therefore the following
derivatives hold true%

\[
D_{q}(e_{q}\left(  z\right)  )=e_{q}\left(  z\right)  \text{ \ \& \ }%
D_{q}(E_{q}\left(  z\right)  )=E_{q}\left(  qz\right)  \text{ \ \& \ }%
D_{q}(\mathcal{E}_{q}\left(  z\right)  )=\frac{\mathcal{E}_{q}\left(
z\right)  +\mathcal{E}_{q}\left(  zq\right)  }{2}\text{\ .}%
\]

\end{example}
\end{lemma}
\end{proposition}

\section{Related q-Bernoulli polynomial}

In this section, we will study the related $q-$Bernoulli polynomials,
$q-$Euler polynomials and $q-$Gennoci polynomials. The discussion of
properties of general $q$-exponential at the previous section, give us the
proper tools to reach to the general properties of these polynomials related
to $\alpha(q,n).$ First, we give the general form of addition theorem.

\begin{proposition}
\label{L:1}\emph{(Addition Theorems)} For all $x,y\in\mathbb{C}$ we have
\begin{align}
\mathfrak{B}_{n,q,\alpha}\left(  x,y\right)   &  =%
{\displaystyle\sum\limits_{k=0}^{n}}
\left(
\begin{array}
[c]{c}%
n\\
k
\end{array}
\right)  _{q}\mathfrak{B}_{k,q,\alpha}\left(  x\oplus_{q}y\right)
^{n-k},\ \ \ \mathfrak{E}_{n,q,\alpha}\left(  x,y\right) \label{be01}\\
&  =%
{\displaystyle\sum\limits_{k=0}^{n}}
\left(
\begin{array}
[c]{c}%
n\\
k
\end{array}
\right)  _{q}\ \mathfrak{E}_{k,q,\alpha}\left(  x\oplus_{q}y\right)
^{n-k},\ \ \ \mathfrak{G}_{n,q,\alpha}\left(  x,y\right)  =%
{\displaystyle\sum\limits_{k=0}^{n}}
\left(
\begin{array}
[c]{c}%
n\\
k
\end{array}
\right)  _{q}\mathfrak{G}_{k,q,\alpha}\left(  x\oplus_{q}y\right)
^{n-k},\nonumber\\
\mathfrak{B}_{n,q,\alpha}\left(  x,y\right)   &  =\sum_{k=0}^{n}\left(
\begin{array}
[c]{c}%
n\\
k
\end{array}
\right)  _{q}\alpha(q,n-k)\mathfrak{B}_{k,q}\left(  x\right)  y^{n-k}%
,\ \mathfrak{E}_{n,q,\alpha}\left(  x,y\right) \label{be2}\\
&  =\sum_{k=0}^{n}\left(
\begin{array}
[c]{c}%
n\\
k
\end{array}
\right)  _{q}\alpha(q,n-k)\mathfrak{E}_{k,q,\alpha}\left(  x\right)
y^{n-k},\ \ \mathfrak{G}_{n,q,\alpha}\left(  x,y\right)  =\sum_{k=0}%
^{n}\left(
\begin{array}
[c]{c}%
n\\
k
\end{array}
\right)  _{q}\alpha(q,n-k)\mathfrak{G}_{k,q,\alpha}\left(  x\right)
y^{n-k}.\nonumber
\end{align}

\end{proposition}

In particular, setting $y=0$ in (\ref{be01}), we get the following formulas
for $q$-Bernoulli, $q$-Euler and $q$-Genocchi polynomials, respectively.%
\begin{align}
\mathfrak{B}_{n,q,\alpha}\left(  x\right)   &  =\sum_{k=0}^{n}\left(
\begin{array}
[c]{c}%
n\\
k
\end{array}
\right)  _{q}\alpha(q,n-k)\mathfrak{B}_{k,q,\alpha}x^{n-k},\ \ \ \mathfrak{E}%
_{n,q,\alpha}\left(  x\right)  =\sum_{k=0}^{n}\left(
\begin{array}
[c]{c}%
n\\
k
\end{array}
\right)  _{q}\alpha(q,n-k)\mathfrak{E}_{k,q,\alpha}x^{n-k},\label{be7}\\
\mathfrak{G}_{n,q,\alpha}\left(  x\right)   &  =\sum_{k=0}^{n}\left(
\begin{array}
[c]{c}%
n\\
k
\end{array}
\right)  _{q}\alpha(q,n-k)\mathfrak{G}_{k,q,\alpha}x^{n-k}. \label{be8}%
\end{align}
Setting $y=1$ in (\ref{be2}), we get%
\begin{align}
\mathfrak{B}_{n,q,\alpha}\left(  x,1\right)   &  =\sum_{k=0}^{n}\left(
\begin{array}
[c]{c}%
n\\
k
\end{array}
\right)  _{q}\alpha(q,n-k)\mathfrak{B}_{k,q,\alpha}\left(  x\right)
,\ \ \ \mathfrak{E}_{n,q,\alpha}\left(  x,1\right)  =\sum_{k=0}^{n}\left(
\begin{array}
[c]{c}%
n\\
k
\end{array}
\right)  _{q}\alpha(q,n-k)\mathfrak{E}_{k,q,\alpha}\left(  x\right)
,\label{be3}\\
\mathfrak{G}_{n,q,\alpha}\left(  x,1\right)   &  =\sum_{k=0}^{n}\left(
\begin{array}
[c]{c}%
n\\
k
\end{array}
\right)  _{q}\alpha(q,n-k)\mathfrak{G}_{k,q,\alpha}\left(  x\right)  .
\label{be4}%
\end{align}
Clearly (\ref{be3}) and (\ref{be4}) are $q$-analogues of%
\[
B_{n}\left(  x+1\right)  =\sum_{k=0}^{n}\left(
\begin{array}
[c]{c}%
n\\
k
\end{array}
\right)  B_{k}\left(  x\right)  ,\ E_{n}\left(  x+1\right)  =\sum_{k=0}%
^{n}\left(
\begin{array}
[c]{c}%
n\\
k
\end{array}
\right)  E_{k}\left(  x\right)  ,\ G_{n}\left(  x+1\right)  =\sum_{k=0}%
^{n}\left(
\begin{array}
[c]{c}%
n\\
k
\end{array}
\right)  G_{k}\left(  x\right)  ,
\]
respectively.we mention that, from the definition of $\mathcal{E}_{q,\alpha
}\left(  t\right)  $, by using the Cauchy product, we reach to $\mathcal{E}%
_{q,\alpha}\left(  tx\right)  .\mathcal{E}_{q,\alpha}\left(  ty\right)
=\sum_{n=0}^{\infty}\frac{t^{n}\left(  x\oplus_{q}y\right)  ^{n}}{[n]_{q}!}$.
putting this equality in $\left(  \text{%
\ref{a1}
}\right)  $ and writting the product of single sums as a double sum, at the
end equating coefficient of $t^{n}$ we lead to the proof of lemma.

\begin{lemma}
The condition \bigskip$\mathcal{E}_{q,\alpha}\left(  -z\right)  =\left(
\mathcal{E}_{q,\alpha}\left(  z\right)  \right)  ^{-1}$ and $\alpha(q,1)=1$
together provides that the odd coefficient of related $q$-Bernoulli numbers
except the first one becomes zero. That means $\mathfrak{B}_{n,q,\alpha}=0$
where $n=2r+1,(r\in%
\mathbb{N}
)$.

\begin{proof}
It follows from the fact that the function
\[
f(t)=%
{\displaystyle\sum\limits_{n=0}^{\infty}}
\frac{\mathfrak{B}_{n,q,\alpha}}{[n]_{q}!}t^{n}-\mathfrak{B}_{1,q,\alpha
}t=\frac{t}{\mathcal{E}_{q,\alpha}\left(  t\right)  -1}+\ \frac{t}{2}%
=\ \frac{t}{2}\left(  \frac{\mathcal{E}_{q}\left(  t\right)  +1}%
{\mathcal{E}_{q}\left(  t\right)  -1}\right)  \ \
\]

is even. we recall that, if \bigskip$\mathcal{E}_{q,\alpha}\left(  -z\right)
=\left(  \mathcal{E}_{q,\alpha}\left(  z\right)  \right)  ^{-1},$ then
$\left(  \text{%
\ref{a2}
}\right)  $ is hold and $\mathfrak{B}_{1,q,\alpha}=-\frac{\alpha(q,2)}%
{\alpha^{2}(q,1)[2]_{q}}.$ Since $\alpha(q,1)=1,$ $\mathfrak{B}_{1,q,\alpha
}=-\frac{1}{2}.$
\end{proof}
\end{lemma}

\begin{lemma}
\label{L:2}If $\alpha(q,n)$ as a parameter of $\mathcal{E}_{q,\alpha}\left(
z\right)  ,$ satisfy $\frac{\alpha(q,n+1)}{\alpha(q,n)}$ $=$ $%
{\displaystyle\sum\limits_{k=0}^{m}}
a_{k}q^{k},$ Then we have%
\begin{align*}
D_{q,x}\mathfrak{B}_{n,q,\alpha}\left(  x\right)   &  =\left[  n\right]  _{q}%
{\displaystyle\sum\limits_{k=0}^{m}}
a_{k}\mathfrak{B}_{n-1,q,\alpha}\left(  xq^{\frac{k}{n}}\right)
,\ \ D_{q,x}\mathfrak{E}_{n,q,\alpha}\left(  x\right)  =\left[  n\right]  _{q}%
{\displaystyle\sum\limits_{k=0}^{m}}
a_{k}\mathfrak{E}_{n-1,q,\alpha}\left(  xq^{\frac{k}{n}}\right)  \ ,\ \\
\ D_{q,x}\mathfrak{G}_{n,q,\alpha}\left(  x\right)   &  =\left[  n\right]
_{q}%
{\displaystyle\sum\limits_{k=0}^{m}}
a_{k}\mathfrak{G}_{n-1,q,\alpha}\left(  xq^{\frac{k}{n}}\right)  \ .
\end{align*}

\end{lemma}

\begin{example}

\end{example}


\begin{thebibliography}{99}                                                                                               %


\bibitem {eXton}Exton, H, q-Hypergeometric Functions and Applications,
Halstead Press, New York, USA, 1983. 

\bibitem {improved}Jan L. Cie%
\'{}%
sli%
\'{}%
nski, Improved q-exponential and q-trigonometric functions,Applied Mathematics
Letters Volume 24, Issue 12, December 2011, Pages 2110--2114.

\bibitem {andrew}G. E. Andrews, R. Askey and R. Roy Special functions, volume
71 of Encyclopedia of Mathematics and its Applications, Cambridge University
Press, Cambridge, 1999.

\bibitem {carlitz1}L. Carlitz, $q$-Bernoulli numbers and polynomials, Duke
Math. J. 15 (1948) 987--1000.

\bibitem {calitz2}L. Carlitz, $q$-Bernoulli and Eulerian numbers, Trans. Amer.
Math. Soc. 76 (1954) 332--350.

\bibitem {calitz3}L. Carlitz, Expansions of $q$-Bernoulli numbers, Duke Math.
J. 25 (1958) 355--364.

\bibitem {sri}H. M. Srivastava, T. Kim and Y. Simsek, q-Bernoulli numbers and
polynomials associated with multiple q-Zeta functions and basic L-series,
Russian J. Math. Phys., 12 (2005), 241-268.

\bibitem {sri1}Q.-M. Luo, H.M. Srivastava, $q$-extensions of some
relationships between the Bernoulli and Euler polynomials, Taiwanese Journal
Math., 15, No. 1, pp. 241-257, 2011.

\bibitem {cenkci1}M. Cenkci and M. Can, Some results on $q$-analogue of the
Lerch Zeta function, Adv. Stud. Contemp. Math., 12 (2006), 213-223.

\bibitem {cenkci2}M. Cenkci, M. Can and V. Kurt, $q$-extensions of Genocchi
numbers, J. Korean Math. Soc., 43 (2006), 183-198.

\bibitem {cenkci3}M. Cenkci, V. Kurt, S. H. Rim and Y. Simsek, On
$(i,q)$-Bernoulli and Euler numbers, Appl. Math. Lett., 21 (2008), 706-711.

\bibitem {choi3}J. Choi, P. J. Anderson and H. M. Srivastava, Carlitz's
$q$-Bernoulli and $q$-Euler numbers and polynomials and a class of $q$-Hurwitz
zeta functions, Appl. Math. Comput., 215 (2009), 1185-1208.

\bibitem {choi2}J. Choi, P. J. Anderson and H. M. Srivastava, Some
$q$-extensions of the Apostol-Bernoulli and the Apostol-Euler polynomials of
order $n$, and the multiple Hurwitz Zeta function, Appl. Math. Comput., 199
(2008), 723-737.

\bibitem {mohammad}N. I. Mahmudov and M. Momenzadeh, On a Class of
q-Bernoulli, q-Euler, and q-Genocchi Polynomials, Hindawi, Abstract and
Applied Analysis, Volume 2014 (2014).
\end{thebibliography}
\end{document}